\documentclass{amsart} 

\usepackage{graphicx} 

\newtheorem{theorem}{Theorem}[section]
\newtheorem{lemma}[theorem]{Lemma}

\newtheorem{corollary}[theorem]{Corollary}

\theoremstyle{definition}
\newtheorem{definition}[theorem]{Definition}

\newcommand{\Z}{\mathbb{Z}}
\newcommand{\R}{\mathbb{R}}

\newcommand{\x}{\mbox{\boldmath $x$}}
\newcommand{\0}{\mbox{\boldmath $0$}}

\begin{document}


\title[11-colored knot diagram with five colors]
{11-colored knot diagram with five colors}

\author{Takuji NAKAMURA}
\address{Department of Engineering Science, 
Osaka Electro-Communication University,
Hatsu-cho 18-8, Neyagawa 572-8530, Japan}
\email{n-takuji@isc.osakac.ac.jp}

\author{Yasutaka NAKANISHI}
\address{Department of Mathematics, Kobe University, 
Rokkodai-cho 1-1, Nada-ku, Kobe 657-8501, Japan}
\email{nakanisi@math.kobe-u.ac.jp}

\author{Shin SATOH}
\address{Department of Mathematics, Kobe University, 
Rokkodai-cho 1-1, Nada-ku, Kobe 657-8501, Japan}
\email{shin@math.kobe-u.ac.jp}

\renewcommand{\thefootnote}{\fnsymbol{footnote}}
\footnote[0]{The third author is partially supported by 
JPSP KAKENHI Grant Number 25400090.}


\renewcommand{\thefootnote}{\fnsymbol{footnote}}





\maketitle


\begin{abstract} 
We prove that any $11$-colorable knot is presented by an $11$-colored diagram 
where exactly five colors of eleven are assigned to the arcs. 
The number five is the minimum for all non-trivially $11$-colored diagrams 
of the knot. 
We also prove a similar result for any $11$-colorable ribbon $2$-knot. 
\end{abstract}


\section{Introduction}\label{sec1}

The $n$-colorability introduced by Fox \cite{Fox} is 
one of the elementary notion in knot theory, 
and its properties have been studied in many papers. 
In $1999$, Harary and Kauffman \cite{HK} defined 
a kind of minimal invariant, ${\rm C}_n(K)$, of an $n$-colorable knot $K$. 
It is essential to consider the case that $n$ is an odd prime; 
in fact, for composite $n$, it is reduced to the cases of odd prime factors of $n$. 
In this case, we can define a modified version by restricting 
``effective" $n$-colorings (cf.~\cite{IM, NNS2}).

Let $p$ be an odd prime. 
A non-trivial $p$-coloring $C$ of a knot diagram $D$ 
is regarded as a non-constant map 
$$C:\mbox{\{arcs of $D$\}}\rightarrow{\Z}/p{\Z}=\{0,1,\dots,p-1\}$$ 
with a certain condition. 
For a $p$-colorable knot $K$, 
the number ${\rm C}_p(K)$ is defined to be the minimum number of 
$\#{\rm Im}(C)$ 
for all non-trivially $p$-colored diagrams $(D,C)$ of $K$. 
This number has been studied in some papers 
\cite{CJZ, GJKLZ, KL, KL2, LM, NNS, Osh, Sai, Sat2}. 
In particular, it is shown in \cite{NNS} that 
$${\rm C}_p(K)\geq \lfloor\log_2 p\rfloor+2$$ 
for any $p$-colorable knot $K$, 
and the equality holds for $p=3,5,7$ 
\cite{Osh, Sat2}.

For $p=11$, we have ${\rm C}_{11}(K)\geq 5$ 
by the above inequality or \cite[Theorem~2.4]{LM}. 
On the other hand, 
it is proved in \cite{CJZ} that 
${\rm C}_{11}(K)\leq 6$. 
If an $11$-colored diagram $(D,C)$ satisfies $\#{\rm Im}(C)=5$, 
then there are two possibilities 
$${\rm Im}(C)=\{1,4,6,7,8\}, \ \{0,4,6,7,8\}$$
up to isomorphisms induced by affine maps of ${\Z}/11{\Z}$. 
This split phenomenon is quite different from 
the cases $p=3,5,7$. 

\begin{theorem}\label{thm11}
Any $11$-colorable knot $K$ satisfies the following. 
\begin{itemize}
\item[{\rm (i)}] 
There is an $11$-colored diagram $(D_1,C_1)$ of $K$ 
with ${\rm Im}(C_1)=\{1,4,6,7,8\}$. 
\item[{\rm (ii)}] 
There is an $11$-colored diagram $(D_2,C_2)$ of $K$ 
with ${\rm Im}(C_2)=\{0,4,6,7,8\}$. 
\end{itemize}
\end{theorem}

We remark that 
these two sets are {\it common} $11$-minimal sufficient sets of colors 
but not {\it universal} ones in the sense of \cite{GJKLZ}. 
By Theorem~\ref{thm11}, we have the following immediately. 

\begin{corollary}\label{cor12}
Any $11$-colorable knot $K$ satisfies 
${\rm C}_{11}(K)=5$. 
\hfill$\Box$
\end{corollary}

This paper is organized as follows. 
In Section~\ref{sec2}, we review the palette graph 
associated with a subset of ${\Z}/p{\Z}$ 
and its fundamental properties. 
In Section~\ref{sec3}, 
we prove Theorem~\ref{thm11}(i). 
The starting point of the proof is a modified version of 
the theorem in \cite{CJZ}: 
For any $11$-colorable knot $K$, 
there is an $11$-colored diagram $(D,C)$ of $K$ 
with ${\rm Im}(C)=\{0,1,4,6,7,8\}$. 
By applying Reidemeister moves to $(D,C)$ suitably, 
we remove the color $0$ from the diagram. 
Sections~\ref{sec4}--\ref{sec6} are devoted 
to proving Theorem~\ref{thm11}(ii). 
We first remove the color $1$ from $(D,C)$ as above 
by allowing the birth of new colors $3$ and $10$ in Section~\ref{sec4}, 
and then remove the colors $10$ and $3$ 
in Sections~\ref{sec5} and \ref{sec6}, respectively. 
In the last section, 
we prove a similar result for an $11$-colorable ribbon  $2$-knot.

\begin{figure}[hbt] 
\setlength{\unitlength}{1mm}
\begin{picture}(110,52)
\put(0,0){\makebox(30,8){$\#{\rm Im}(C)=5$}}
\put(0,20){\makebox(30,8){$\#{\rm Im}(C)=6$}}
\put(0,40){\makebox(30,8){$\#{\rm Im}(C)=7$}}
\put(39,0){\framebox(22,8){$\{1,4,6,7,8\}$}}
\put(79,0){\framebox(22,8){$\{0,4,6,7,8\}$}}
\put(37,20){\framebox(26,8){$\{0,1,4,6,7,8\}$}}
\put(38,21){\framebox(24,6){}}
\put(77,20){\framebox(26,8){$\{0,3,4,6,7,8\}$}}
\put(75,40){\framebox(30,8){$\{0,3,4,6,7,8,10\}$}}
\put(50,20){\vector(0,-1){12}}
\put(90,20){\vector(0,-1){12}}
\put(90,40){\vector(0,-1){12}}
\put(50,44){\vector(1,0){25}}
\put(50,44){\line(0,-1){16}}
\put(52,16){Remove $0$}
\put(52,11){(Section 3)}
\put(92,16){Remove $3$}
\put(92,11){(Section 6)}
\put(92,36){Remove $10$}
\put(92,31){(Section 5)}
\put(52,30){(Section 4)}
\put(52,35){with making $3$ and $10$}
\put(52,40){Remove $1$}
\end{picture}
\end{figure}


\section{Preliminaries}\label{sec2} 

Throughout this section, 
$p$ denotes an odd prime. 

\begin{definition}\label{def21}
Let $S$ be a subset of ${\Z}/p{\Z}$. 
The {\it palette graph} $G(S)$ of $S$ 
is a simple graph such that 
\begin{itemize}
\item[(i)] 
the vertex set of $G(S)$ is $S$, and 
\item[(ii)] 
two vertices $a$ and $b\in S$ are connected by an edge 
if and only if $\frac{a+b}{2}\in S$. 
\end{itemize}
By assigning $\frac{a+b}{2}$ to every edge joining $a$ and $b$, 
we regard $G(S)$ as a labeled graph. 
Such an edge is denoted by $\{a|\frac{a+b}{2}|b\}$. 
\end{definition} 

\begin{definition}\label{def22}
For two subsets $S$ and $S'\subset{\Z}/p{\Z}$, 
the palette graphs $G(S)$ and $G(S')$  are said to be 
{\it isomorphic} if there is a bijection 
$f:S\rightarrow S'$ such that 
$\frac{a+b}{2}\in S$ if and only if $\frac{f(a)+f(b)}{2}\in S'$. 
We denote it by $G(S)\cong G(S')$. 
\end{definition} 

\begin{lemma}\label{lem23}
If $S\subset S'\subset{\Z}/p{\Z}$, then $G(S)$ is a subgraph of $G(S')$, 
which is obtained from $G(S')$ 
by deleting the vertices in $S'\setminus S$ and 
the edges whose labels belong to $S'\setminus S$. 
\end{lemma}

\begin{proof}
This follows from definition immediately. 
\end{proof}

\begin{theorem}[\cite{NNS}]\label{thm24}
If the palette graph $G(S)$ is connected with $\#S>1$, 
then we have 
$\#S\geq\lfloor\log_2 p\rfloor+2$. 
\hfill$\Box$
\end{theorem}

\begin{lemma}\label{lem25}
Let $S$ be a subset of ${\Z}/p{\Z}$ 
such that $G(S)$ is connected with  
$\#S=\lfloor\log_2 p\rfloor+2$. 
Put 
$U=\{S'\subset{\Z}/p{\Z}|G(S')\cong G(S)\}$. 
Then we have 
$\#U=p(p-1)$. 
\end{lemma}

\begin{proof}
Let $T$ be a maximal tree of $G(S)$. 
Let $v_1,v_2,\dots,v_k$ be the vertices of $T$, 
and $e_1,e_2,\dots,e_{k-1}$ the edges of $T$, 
where $k=\#S=\lfloor\log_2 p\rfloor+2$. 
Let $A=(a_{ij})$ be the $(k-1)\times k$ matrix 
with ${\Z}$-entries defined by  
$$a_{ij}=\left\{
\begin{array}{rl} 
1 & \mbox{($e_i$ is incident to $v_j$)}, \\
-2 & \mbox{(the label of $e_i$ is $v_j$)}, \\
0 & \mbox{(otherwise)}. 
\end{array}\right.$$
Let $A'$ be the $(k-1)\times(k-1)$ matrix 
obtained from $A$ by deleting the $k$th column. 
It is known in \cite{NNS} that 
\begin{itemize}
\item[(i)] 
${\rm det}(A')$ is odd, 
\item[(ii)] 
$|{\rm det}(A')|<2^{k-1}$, and 
\item[(iii)] 
${\rm det}(A')$ is divisible by $p$. 
\end{itemize}
Since 
$2^{k-2}<p\leq|{\rm det}(A')|<2^{k-1}$, 
we have $|{\rm det}(A')|=p$. 
This implies that the corank of $A$ with ${\Z}/p{\Z}$-entries is exactly equal to $2$.

Let $V=\{{\x}|A{\x}\equiv{\0}\mbox{ (mod $p$)}\}$ denote the solution space. 
By the above argument, we have 
$$V=\{\lambda\cdot{}^{\rm t}(v_1,v_2,\dots,v_k)+\mu\cdot {}^{\rm t}(1,1,\dots,1)|
\lambda,\mu\in{\Z}/p{\Z}\}.$$
Since the elements of $U$ are identified with 
the vectors of $V$ whose entries are all distinct. 
Such a vector is obtained by the condition $\lambda\not\equiv 0$ (mod~$p$). 
Therefore, we have $\#U= p(p-1)$. 
\end{proof}

\begin{theorem}\label{thm26}
Let $S$ and $S'$ be subsets of ${\Z}/p{\Z}$. 
Suppose that $G(S)$ and $G(S')$ are connected 
with $\#S=\#S'=\lfloor\log_2p\rfloor+2$. 
Then the following are equivalent. 
\begin{itemize}
\item[{\rm (i)}] 
The palette graphs $G(S)$ and $G(S')$ are isomorphic. 
\item[{\rm (ii)}] 
There exist 
$\alpha\not\equiv 0$ and $\beta\in{\Z}/p{\Z}$ 
such that the affine map $f(x)=\alpha x+\beta$ 
satisfies $f(S)=S'$. 
\end{itemize}
\end{theorem}

\begin{proof} 
\underline{(ii)$\Rightarrow$(i).} 
Since $\alpha\not\equiv 0$ (mod~$p$), 
$f:S\rightarrow S'$ is a bijection. 
Furthermore, 
$\frac{a+b}{2}\in S$ holds if and only if $f(\frac{a+b}{2})=\frac{f(a)+f(b)}{2}\in f(S)=S'$ holds. 

\underline{(i)$\Rightarrow$(ii).} 
By the above argument, 
we have 
$$U\supset\{f(S)|f(x)=\alpha x+\beta, \alpha\not\equiv 0, \beta\in{\Z}/p{\Z}\},$$
where $U$ is the set in Lemma~\ref{lem25}. 
Since these two sets have the same number of elements by Lemma~\ref{lem25}, 
they are the same set. 
\end{proof}

Let $D$ be a diagram of a knot $K$. 
We regard $D$ as a disjoint union of arcs 
whose endpoints are under-crossings. 
Fox \cite{Fox} introduced the notion of $p$-colorings: 
A map $C:\{\mbox{arcs of $D$}\}\rightarrow{\Z}/p{\Z}$ 
is a {\it $p$-coloring} if $a+b\equiv 2c$ (mod~$p$) 
holds at every crossing, 
where $a$ and $b$ are the elements assigned to the under-arcs by $C$, 
and $c$ is the one to the over-arc. 
The triple $\{a|c|b\}$ 
is called the {\it color} of the crossing. 
The assigned element of an arc of $D$ is called the {\it color} of the arc. 
If the color of an arc is $a$, 
then the arc is called an {\it $a$-arc}. 

In a $p$-colored diagram $(D,C)$, 
the crossing of color $\{a|a|a\}$ is called {\it trivial}, 
and otherwise {\it non-trivial}. 
If $C$ is a constant map, 
it is called a {\it trivial $p$-coloring}, 
and otherwise, {\it non-trivial}. 
In other words, 
a $p$-coloring $C$ is non-trivial 
if and only if $\#{\rm Im}(C)>1$. 
If a knot $K$ admits a non-trivially $p$-colored diagram $(D,C)$, 
$K$ is called {\it $p$-colorable}. 

For a $p$-colorable knot $K$, 
we denote by ${\rm C}_p(K)$ 
the minimum number of $\#{\rm Im}(C)$ 
for all non-trivially $p$-colored diagram $(D,C)$ of $K$ \cite{HK}. 
For the study of this number, 
it is helpful to use the palette graph $G({\rm Im}(C))$ 
of the image ${\rm Im}(C)\subset{\Z}/p{\Z}$ 
in the following sense. 

\begin{lemma}\label{lem27}
If $\{a|c|b\}$ is a non-trivial color of a crossing of a $p$-colored diagram $(D,C)$, 
then the palette graph $G({\rm Im}(C))$ has an edge $\{a|c|b\}$. 
\end{lemma}

\begin{proof}
Since $a+b\equiv 2c$ (mod~$p$) holds, 
the lemma follows by definition. 
\end{proof}

\begin{lemma}\label{lem28}
The palette graph $G({\rm Im}(C))$ of 
a $p$-colored diagram $(D,C)$ of a knot is connected. 
\end{lemma}

\begin{proof}
Let $a$ and $b$ be vertices of $G({\rm Im}(C))$. 
By definition, we have an $a$-arc and a $b$-arc of $D$. 
Since $D$ is a diagram of a knot (not a link), 
we can walk along $D$ from the $a$-arc to the $b$-arc. 
Let $\{a_i|c_i|a_{i+1}\}$ $(1\leq i\leq k-1)$ be 
the colors of non-trivial under-crossings on the path 
such that $a_1=a$ and $a_k=b$. 
Then the vertices $a$ and $b$ in the palette graph are 
connected by a sequence of edges 
$\{a_i|c_i|a_{i+1}\}$ $(1\leq i\leq k-1)$. 
\end{proof}

\begin{theorem}[\cite{NNS}]\label{thm29}
Any non-trivial $p$-colored diagram $(D,C)$ of a knot 
satisfies $\#{\rm Im}(C)\geq\lfloor\log_2p\rfloor+2$. 
Therefore, we have 
${\rm C}_p(K)\geq\lfloor\log_2p\rfloor+2$ 
for any $p$-colorable knot $K$. 
\end{theorem}

\begin{proof}
This follows from Theorem~\ref{thm24} and Lemma~\ref{lem28}. 
\end{proof}

\begin{lemma}\label{lem210}
Let $(D,C)$ be a non-trivially $p$-colored diagram of a knot $K$, 
and $f:{\Z}/p{\Z}\rightarrow{\Z}/p{\Z}$ 
an affine map defined by $f(x)=\alpha x+\beta$ 
with $\alpha\not\equiv 0$ and $\beta\in{\Z}/p{\Z}$. 
Then there is a non-trivially $p$-colored diagram 
$(D,C')$ of $K$ such that 
${\rm Im}(C')=f({\rm Im}(C))$. 
\end{lemma}

\begin{proof}
It is easy to see that the composition 
$C'=f\circ C$ is also a non-trivial $p$-coloring of $D$. 
\end{proof}

Now, we consider the case $p=11$. 
By Theorem~\ref{thm24}, 
if the palette graph $G(S)$ of a subset $S\subset {\Z}/11{\Z}$ 
is connected with $\#S>1$, 
then $\#S\geq 5$.

\begin{theorem}[{\cite[Theorem~12]{GJKLZ}}]\label{thm211}
Let $S$ be a subset of ${\Z}/11{\Z}$. 
If the palette graph $G(S)$ is connected 
with $\#S=5$, 
then $G(S)$ is isomorphic to $G(\{1,4,6,7,8\})$ or $G(\{0,4,6,7,8\})$ 
as shown in {\rm Figure~\ref{fig2-01}}. 
\hfill$\Box$
\end{theorem}

\begin{figure}[htb]
\begin{center}
\includegraphics[bb=0 0 248 91]{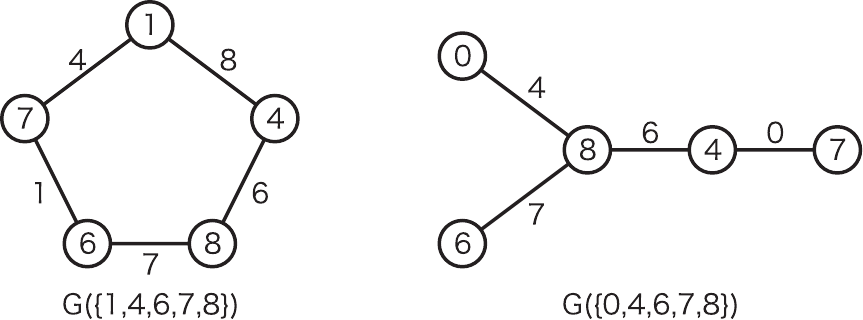}
\caption{}
\label{fig2-01}
\end{center}
\end{figure}

By Theorem~\ref{thm29} or \cite[Theorem~2.4]{LM}, 
we have ${\rm C}_{11}(K)\geq 5$. 
The following theorem implies that ${\rm C}_{11}(K)=5$ or $6$. 

\begin{theorem}[\cite{CJZ}]\label{thm212}
For any $11$-colorable knot $K$, 
there is a non-trivially $11$-colored diagram $(D,C)$ of $K$ 
with ${\rm Im}(C)\subset\{0,1,4,6,7,8\}$. 
\hfill$\Box$
\end{theorem}

Figure~\ref{fig2-02} shows the palette graph 
$G(\{0,1,4,6,7,8\})$. 
By Lemma~\ref{lem23}, 
the two graphs in Theorem~\ref{thm211} 
are obtained from this graph by deleting 
the vertex $a$ and the edges labeled $a$ 
for $a=0,1$, respectively. 

\begin{figure}[htb]
\begin{center}
\includegraphics[bb=0 0 86 91]{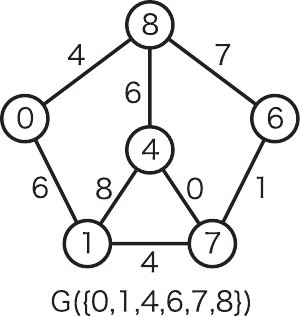}
\caption{}
\label{fig2-02}
\end{center}
\end{figure}

It is useful for our argument to modify Theorem~\ref{thm212} slightly as follows.

\begin{lemma}\label{lem213}
For any $11$-colorable knot $K$, 
there is an $11$-colored diagram $(D,C)$ of $K$ 
with ${\rm Im}(C)=\{0,1,4,6,7,8\}$. 
\end{lemma}

\begin{proof}
We may assume that 
$(D,C)$ satisfies Theorem~\ref{thm212}; that is, 
it is a non-trivially $11$-colored diagram 
with ${\rm Im}(C)\subset \{0,1,4,6,7,8\}$. 
We remark that $\#{\rm Im}(C)\geq 5$ 
by Theorem~\ref{thm29}. 

\underline{$4,6,7\in{\rm Im}(C)$.} 
Assume that $4\not\in{\rm Im}(C)$. 
It follows that ${\rm Im}(C)=\{0,1,6,7,8\}$. 
The palette graph 
$G({\rm Im}(C))$ is as shown in the left of Figure~\ref{fig2-03} 
by Lemma~\ref{lem23}, 
which contradicts to the connectivity in Lemma~\ref{lem28}. 
We can also prove $6,7\in{\rm Im}(C)$ 
by a similar argument. 
See the center and right of the figure.

\begin{figure}[htb]
\begin{center}
\includegraphics[bb=0 0 302 90]{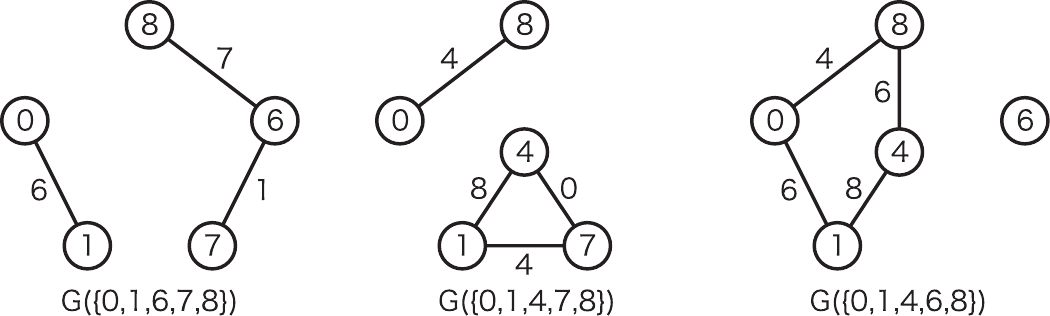}
\caption{}
\label{fig2-03}
\end{center}
\end{figure}

\underline{$0\in{\rm Im}(C)$.} 
Assume that $0\not\in{\rm Im}(C)$. 
It follows that ${\rm Im}(C)=\{1,4,6,7,8\}$ 
and its palette graph is 
as shown in the left of Figure~\ref{fig2-01}. 
Then we see that $(D,C)$ has a crossing of color $\{6|1|7\}$ or $\{1|8|4\}$. 
In fact, if we delete the corresponding edges both, 
the resulting graph becomes disconnected.  
By deforming the diagram near these crossings as shown in Figure~\ref{fig2-04}, 
we can produce a $0$-arc. 
We replace the original diagram with the new one as $(D,C)$.

\begin{figure}[htb]
\begin{center}
\includegraphics[bb=0 0 320 48]{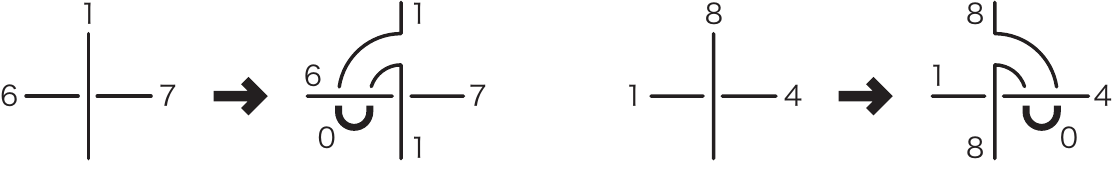}
\caption{}
\label{fig2-04}
\end{center}
\end{figure}

\underline{$1\in{\rm Im}(C)$.} 
Assume that $1\not\in{\rm Im}(C)$. 
Then we have ${\rm Im}(C)=\{0,4,6,7,8\}$ 
and its palette graph is as shown the right of Figure~\ref{fig2-01}. 
Since $(D,C)$ must have a crossing of color $\{0|4|8\}$ 
by a similar reason to the above case, 
we deform the diagram near the crossing 
to make a $1$-arc. 
See Figure~\ref{fig2-05}.

\begin{figure}[htb]
\begin{center}
\includegraphics[bb=0 0 139 46]{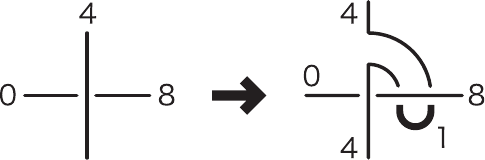}
\caption{}
\label{fig2-05}
\end{center}
\end{figure}

\underline{$8\in{\rm Im}(C)$.} 
Assume that $8\not\in{\rm Im}(C)$. 
Then we have ${\rm Im}(C)=\{0,1,4,6,7\}$ 
and its palette graph is as shown in the left of Figure~\ref{fig2-06}. 
We remark that the map $f:{\Z}/11{\Z}\rightarrow{\Z}/11{\Z}$ 
defined by $f(x)=7x+6$ induces 
the isomorphism between 
$G(\{0,4,6,7,8\})$ and $G(\{0,1,4,6,7\})$. 
The existence of such a map is guaranteed by Theorem~\ref{thm26}. 
Since $(D,C)$ has a crossing of color $\{4|0|7\}$, 
we deform the diagram near the crossing 
as shown in the right of the figure 
so that we obtain an $8$-arc. 
\end{proof}

\begin{figure}[htb]
\begin{center}
\includegraphics[bb=0 0 310 69]{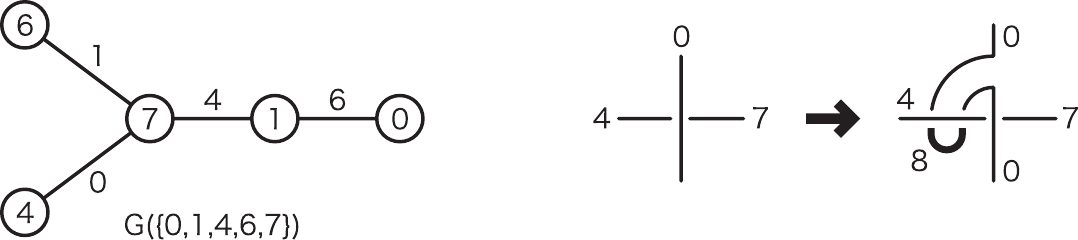}
\caption{}
\label{fig2-06}
\end{center}
\end{figure}


\section{Proof of Theorem~\ref{thm11}{\rm (i)}}\label{sec3}

\begin{lemma}\label{lem31}
For any $11$-colorable knot $K$, 
there is an $11$-colored diagram $(D,C)$ of $K$ 
such that 
\begin{itemize}
\item[{\rm (i)}] 
${\rm Im}(C)=\{0,1,4,6,7,8\}$, and 
\item[{\rm (ii)}] 
there is no crossing of color $\{*|0|*\}$. 
\end{itemize}
\end{lemma}

\begin{proof}
We may assume that $(D,C)$ satisfies Lemma~\ref{lem213}. 
There are two types of crossings of $(D,C)$ 
whose over-arc is a $0$-arc; 
that is, $\{0|0|0\}$ and $\{4|0|7\}$. 
In fact, in the palette graph $G(\{0,1,4,6,7,8\})$, 
the only edge labeled $0$ connects $4$ and $7$. 

First, we assume that $(D,C)$ has crossings of color $\{4|0|7\}$. 
By deforming the diagram near the crossings as shown in Figure~\ref{fig3-01}, 
we can eliminate all the crossings of color $\{4|0|7\}$. 
We remark that the set of colors which are appeared in the diagram 
does not change.

\begin{figure}[htb]
\begin{center}
\includegraphics[bb=0 0 154 54]{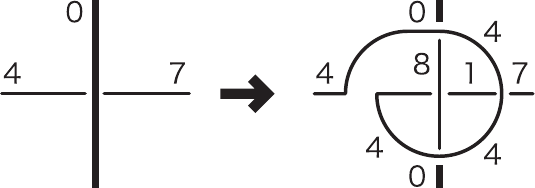}
\caption{}
\label{fig3-01}
\end{center}
\end{figure}

Next, we assume that $(D,C)$ has a crossing of color $\{0|0|0\}$, say $x$. 
Walking along the diagram from $x$, 
let $y$ be the non-trivial crossing which we meet first. 
If there are crossings of color $\{0|0|0\}$ between $x$ and $y$, 
we replace the original $x$ with the nearest one to $y$. 
Therefore, we may assume that 
there is no crossing between $x$ and $y$. 

There are two cases with respect to the color of $y$. 
In fact, in the palette graph $G(\{0,1,4,6,7,8\})$, 
there are two edges incident to the vertex $0$, 
which implies that the color of $y$ is $\{0|6|1\}$ or $\{0|4|8\}$. 
In each case, we deform the diagram $(D,C)$ near $x$ and $y$ 
as shown in Figure~\ref{fig3-02}, 
so that the number of crossings of $\{0|0|0\}$ is decreased. 
By repeating this process, 
we obtain a diagram with no crossing of $\{0|0|0\}$ finally. 
\end{proof}

\begin{figure}[htb]
\begin{center}
\includegraphics[bb=0 0 342 55]{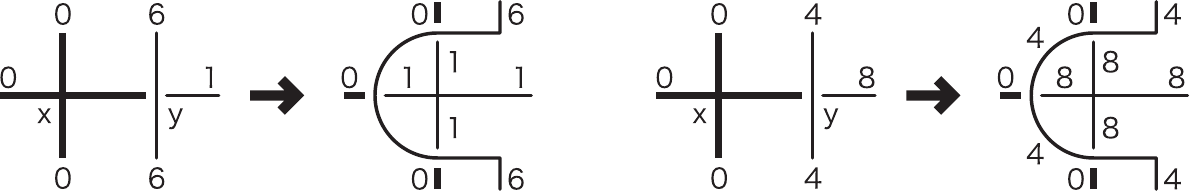}
\caption{}
\label{fig3-02}
\end{center}
\end{figure}

\begin{proof}[Proof of {\rm Theorem~\ref{thm11}(i)}.] 
We may assume that $(D,C)$ satisfies Lemma~\ref{lem31}. 
If there is a $0$-arc, 
it is not an over-arc of any crossing, 
and its endpoints are the under-crossings 
of color $\{0|4|8\}$ or $\{0|6|1\}$. 
In fact, there are two edges incident to the vertex $0$ 
in $G(\{1,4,6,7,8\})$. 
We have three cases with respect to 
the colors of the crossings of the endpoints of a $0$-arc; 
\begin{itemize}
\item[(i)] 
$\{0|4|8\}$ and $\{0|6|1\}$, 
\item[(ii)] 
$\{0|4|8\}$ both, 
and 
\item[(iii)] 
$\{0|6|1\}$ both. 
\end{itemize}

For the case (i), 
we deform the $6$-arc 
over the crossing of $\{0|4|8\}$ 
to eliminate the $0$-arc. 
See the top of Figure~\ref{fig3-03}. 
For the case (ii), 
we deform one of the crossings of color $\{0|4|8\}$ 
as shown in the figure so that 
we reduce this case to (i). 
Similarly, for the case (iii), 
we deform one of the crossings of color $\{0|6|1\}$ 
as shown in the figure so that 
we reduce this case to (i). 
See the bottom of the figure. 
\end{proof}

\begin{figure}[htb]
\begin{center}
\includegraphics[bb=0 0 325 244]{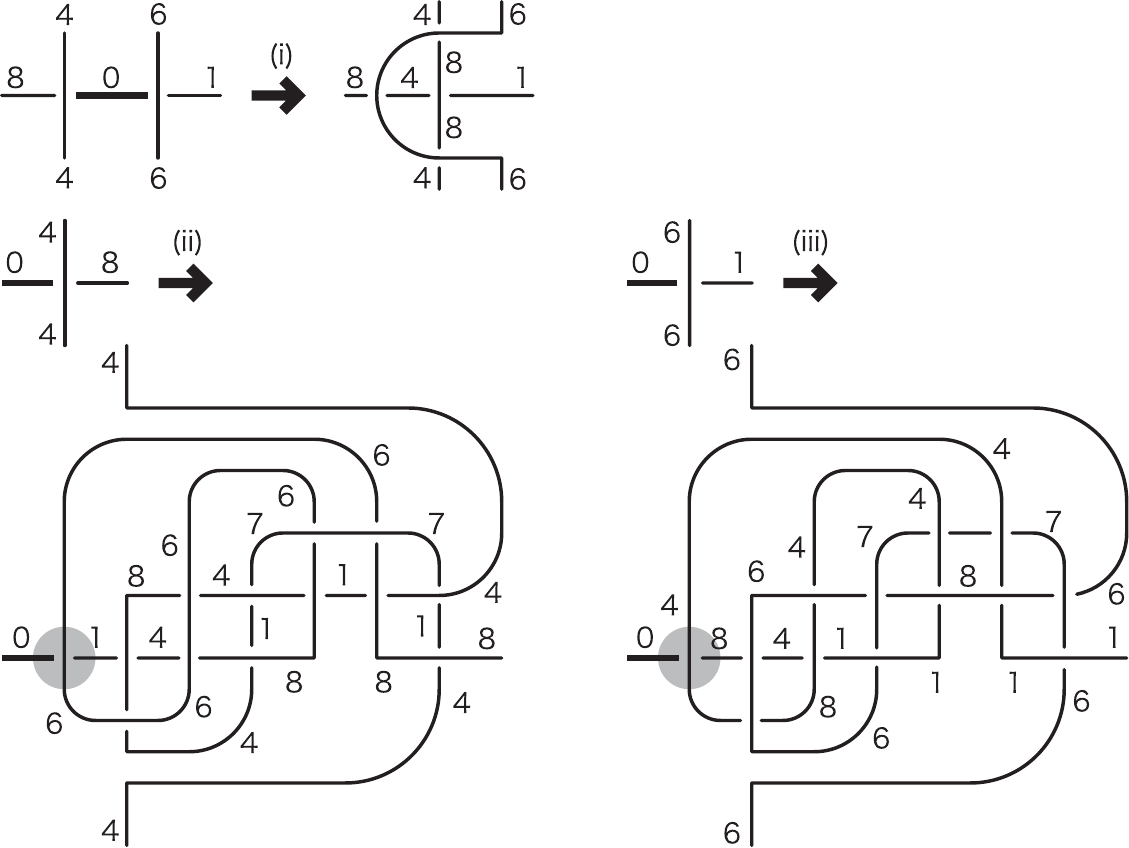}
\caption{}
\label{fig3-03}
\end{center}
\end{figure}

\begin{corollary}\label{cor32}
For any $11$-colorable knot $K$ and $a\not\equiv b\in{\Z}/11{\Z}$, 
there is an $11$-colored diagram $(D,C)$ of $K$ 
with $${\rm Im}(C)=\{a,b,3a+9b,6a+6b, 10a+2b\}.$$
\end{corollary}

\begin{proof}
Let $f:{\Z}/11{\Z}\rightarrow{\Z}/11{\Z}$ be the affine map  
defined by $f(x)=4(b-a)(x-1)+a$. 
Since the map $f$ satisfies 
$$f(1)=a, \ f(4)=b, \ f(6)=3a+9b, \ f(7)=10a+2b, \ f(8)=6a+6b,$$
we have the conclusion by Lemma~\ref{lem210} and Theorem~\ref{thm11}(i). 
\end{proof}


\section{Proof of Theorem~\ref{thm11}{\rm (ii)}--Part I}\label{sec4}

\begin{lemma}\label{lem41}
For any $11$-colorable knot $K$, 
there is an $11$-colored diagram $(D,C)$ of $K$ 
such that 
\begin{itemize}
\item[{\rm (i)}] 
${\rm Im}(C)=\{0,1,4,6,7,8\}$, and 
\item[{\rm (ii)}] 
there is no crossing of color $\{6|6|6\}$. 
\end{itemize}
\end{lemma}

\begin{proof}
We may assume that $(D,C)$ satisfies Lemma~\ref{lem213}. 
Assume that $(D,C)$ has a crossing of color $\{6|6|6\}$, say $x$. 
Walking along the diagram from $x$, 
let $y$ be the first non-trivial {\it under}-crossing. 
If there are crossings of color $\{6|6|6\}$ between $x$ and $y$, 
then we replace the original $x$ with the nearest one to $y$. 
Then we have the following: 
\begin{itemize} 
\item[(i)] 
There is no crossing of $\{6|6|6\}$ 
between $x$ and $y$ by assumption. 
\item[(ii)] 
Every crossing between $x$ and $y$ is of color 
$\{0|6|1\}$ or $\{4|6|8\}$; 
for there are exactly two edges labeled $6$ 
in the palette graph $G(\{0,1,4,6,7,8\})$. 
\item[(iii)] 
The color of $y$ is $\{6|1|7\}$ or $\{6|7|8\}$; 
for there are exactly two edges incident to the vertex $6$ 
in the palette graph, 
which are labeled $1$ and $7$, respectively.  
\end{itemize}

Assume that there are crossings between $x$ and $y$. 
Let $z$ be the nearest crossing to $x$ among them. 
We deform the diagram near $x$ and $z$ 
as shown in the upper row of Figure~\ref{fig4-01}, 
so that the number of crossings between $x$ and $y$ is decreased. 
By repeating this process, 
we may assume that there is no crossing between $x$ and $y$. 
Then we deform the diagram near $x$ and $y$ 
as shown in the lower row of the figure 
to eliminate the color $\{6|6|6\}$. 
By repeating this process, 
we obtain a diagram with no crossing of $\{6|6|6\}$ finally. 
\end{proof}

\begin{figure}[htb]
\begin{center}
\includegraphics[bb=0 0 342 127]{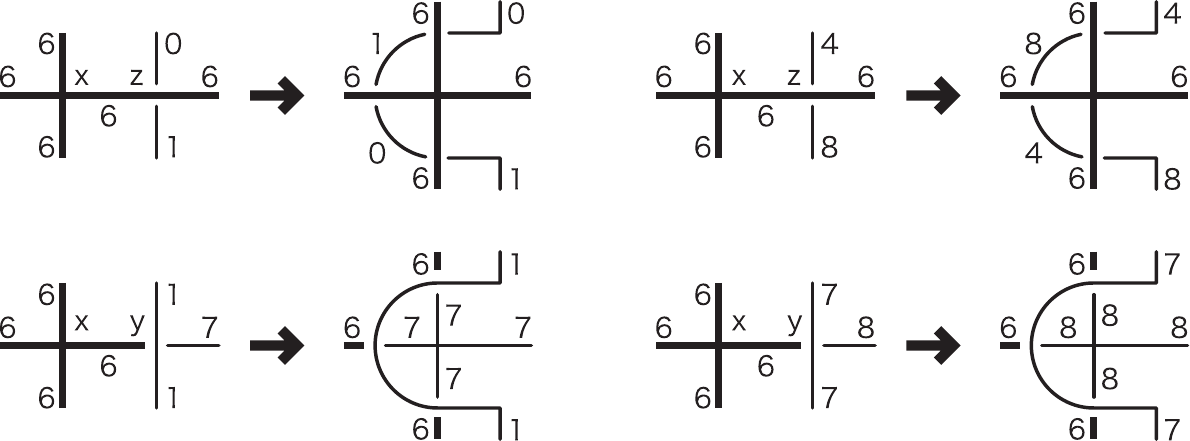}
\caption{}
\label{fig4-01}
\end{center}
\end{figure}

\begin{lemma}\label{lem42}
For any $11$-colorable knot $K$, 
there is an $11$-colored diagram $(D,C)$ of $K$ 
such that 
\begin{itemize}
\item[{\rm (i)}] 
${\rm Im}(C)=\{0,1,4,6,7,8\}$, and 
\item[{\rm (ii)}] 
there is no crossing of color $\{1|1|1\}$ or $\{6|6|6\}$. 
\end{itemize}
\end{lemma}

\begin{proof}
We may assume that $(D,C)$ satisfies Lemma~\ref{lem41}. 
Assume that $(D,C)$ has a crossing of color $\{1|1|1\}$, 
say $x$. 
Walking along the diagram from $x$, 
let $y$ be the first non-trivial crossing. 
If there are crossing of $\{1|1|1\}$ between $x$ and $y$, 
then we replace the original $x$ with the nearest one to $y$.

In the palette graph $G(\{0,1,4,6,7,8\})$, 
there are exactly three edges incident to the vertex $1$ 
whose labels are $4$, $6$, and $8$, 
and there is only one edge whose label is $1$. 
Therefore, 
the color of the crossing $y$ is $\{1|4|7\}$, 
$\{1|6|0\}$, $\{1|8|4\}$, or $\{6|1|7\}$. 

We deform the diagram near $x$ and $y$ as shown in 
Figure~\ref{fig4-02} 
so that the number of crossings of color $\{1|1|1\}$ 
is decreased. 
By repeating this process, 
we obtain a diagram with no crossing of $\{1|1|1\}$. 
\end{proof}

\begin{figure}[htb]
\begin{center}
\includegraphics[bb=0 0 351 127]{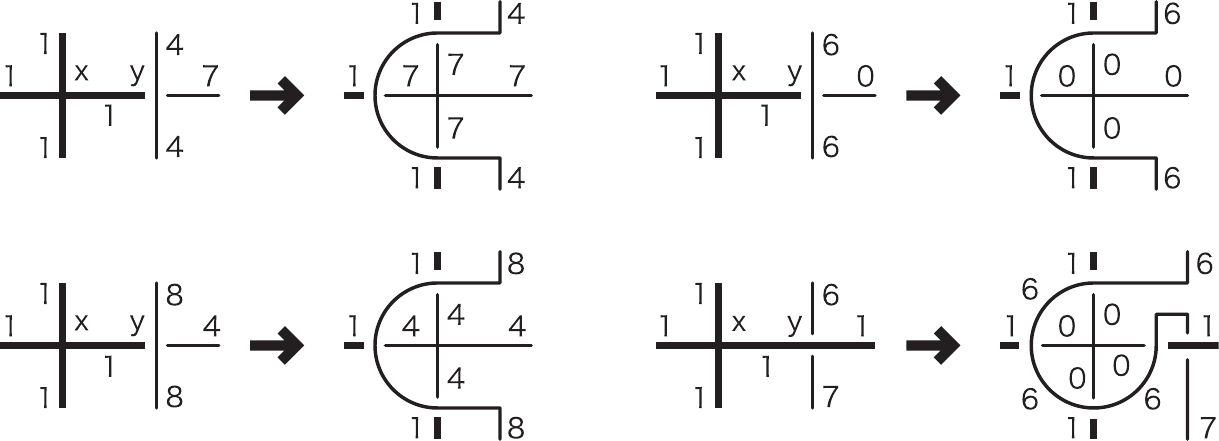}
\caption{}
\label{fig4-02}
\end{center}
\end{figure}

\begin{lemma}\label{lem43}
For any $11$-colorable knot $K$, 
there is a non-trivially $11$-colored diagram $(D,C)$ of $K$ 
such that 
\begin{itemize}
\item[{\rm (i)}] 
${\rm Im}(C)=\{0,1,4,6,7,8\}$, and   
\item[{\rm (ii)}] 
there is no crossing of color $\{*|1|*\}$ or $\{6|6|6\}$. 
\end{itemize}
\end{lemma}

\begin{proof}
We may assume that $(D,C)$ satisfies Lemma~\ref{lem42}. 
Assume that $(D,C)$ has a crossing of color $\{*|1|*\}$. 
Since there is only one edge labeled $1$ 
in the palette graph $G(\{0,1,4,6,7,8\})$, 
the color of the corresponding crossing is $\{6|1|7\}$. 

There is a $4$-arc in $(D,C)$. 
We will pull the $4$-arc toward each crossing of $\{6|1|7\}$. 
In the process, we can assume that 
the $4$-arc crosses over several arcs whose colors are $0,1,4, 7,8$ 
missing $6$. 
In fact, since there is no crossing of $\{6|6|6\}$, 
the set of $6$-arcs is a disjoint union of intervals in the plane, 
and the complement in the plane is connected. 
When the $4$-arc crosses over an $a$-arc for $a=0,1,4,7,8$, 
we have a pair of new crossings of color 
$$\{a|4|8-a\}=\{0|4|8\}, \ \{1|4|7\}, \ \{4|4|4\}, \ \{7|4|1\}, \ \{8|4|0\},$$
respectively. 
See the left of Figure~\ref{fig4-03}. 
We remark that any vertex of the palette graph $G(\{0,1,4,6,7,8\})$ 
other than $6$ is $4$ itself or incident to an edge labeled $4$.

\begin{figure}[htb]
\begin{center}
\includegraphics[bb=0 0 343 144]{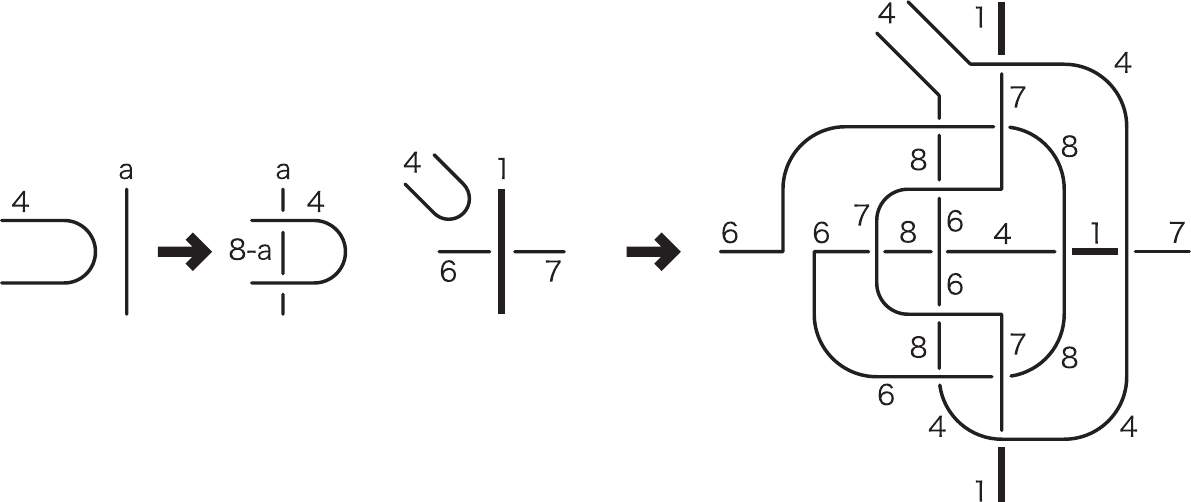}
\caption{}
\label{fig4-03}
\end{center}
\end{figure}

By deforming the diagram near every crossing of $\{6|1|7\}$ 
with a $4$-arc as shown in the right of the figure, 
we obtain a diagram with no crossing of $\{6|1|7\}$. 
Then the arcs in the obtained diagram are colored by $0,1,4,6,7,8$ and 
there is no crossing of $\{*|1|*\}$ or $\{6|6|6\}$. 
\end{proof}

\begin{lemma}\label{lem44}
For any $11$-colorable knot $K$, 
there is an $11$-colored diagram $(D,C)$ of $K$ 
such that 
\begin{itemize}
\item[{\rm (i)}] 
${\rm Im}(C)=\{0,3,4,6,7,8,10\}$,  
\item[{\rm (ii)}] 
there is no crossing of color $\{6|6|6\}$, and 
\item[{\rm (iii)}] 
if $\{a|b|c\}$ is the color of a crossing 
and at least one of $a,b,c$ is $3$ or $10$, 
then it is one of 
$$\{0|3|6\}, \ \{0|7|3\}, \ \{3|0|8\}, \ 
\{4|7|10\}, \ \{7|3|10\}, \ \{3|3|3\}.$$
\end{itemize}
\end{lemma}

\begin{proof}
We may assume that $(D,C)$ satisfies Lemma~\ref{lem43}. 
Since there are three edges incident to the vertex $1$ 
in the palette graph $G(\{0,1,4,6,7,8\})$, 
every crossing with a $1$-arc is of color 
$\{1|4|7\}$, $\{1|6|0\}$, or $\{1|8|4\}$. 
If there is a crossing of $\{1|8|4\}$, 
we deform the $4$-arc near the crossing 
as shown in Figure~\ref{fig4-04} 
to replace the crossing with the one of color $\{1|4|7\}$. 
Therefore, 
we may assume that there is no crossing of $\{1|8|4\}$.

\begin{figure}[htb]
\begin{center}
\includegraphics[bb=0 0 162 55]{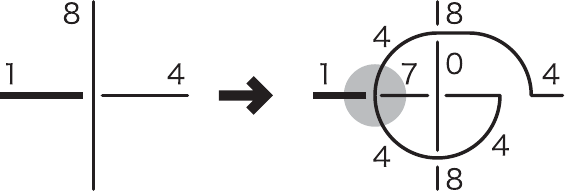}
\caption{}
\label{fig4-04}
\end{center}
\end{figure}

There is a $0$-arc in $(D,C)$. 
We will pull the $0$-arc toward each crossing of $\{1|4|7\}$. 
In the process, 
we can assume that the $0$-arc crosses 
over several $a$-arcs 
for $a\in\{0,1,4,7,8\}$ missing $6$ 
by the same reason in the proof of Lemma~\ref{lem43}; 
that is, there is no crossing of $\{6|6|6\}$. 
When the $0$-arc crosses over an $a$-arc, 
we have a pair of new crossings of color 
$$\{a|0|-a\}=\{0|0|0\},\ \{1|0|10\}, \ \{4|0|7\}, \ \{7|0|4\}, \ \{8|0|3\},$$ 
respectively. 
We remark that 
the new colors $3$ and $10$ appear 
at the crossings of $\{1|0|10\}$ and $\{3|0|8\}$. 
See Figure~\ref{fig4-05}.

\begin{figure}[htb]
\begin{center}
\includegraphics[bb=0 0 244 45]{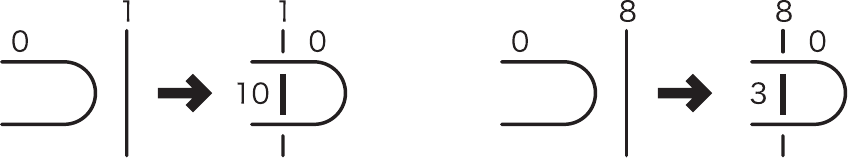}
\caption{}
\label{fig4-05}
\end{center}
\end{figure}

By deforming the diagram near every crossing of $\{1|4|7\}$ 
with a $0$-arc as shown in Figure~\ref{fig4-06}, 
we remove all the crossings of $\{1|4|7\}$ and 
produce the color $10$ at the crossings of 
$\{1|0|10\}$ and $\{4|7|10\}$.

\begin{figure}[htb]
\begin{center}
\includegraphics[bb=0 0 173 57]{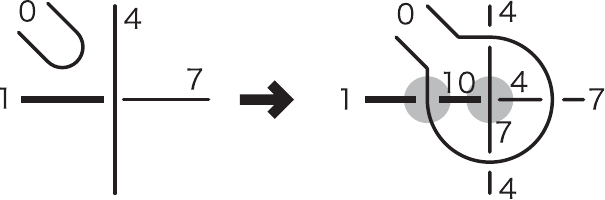}
\caption{}
\label{fig4-06}
\end{center}
\end{figure}

There is a $7$-arc in $(D,C)$. 
We will pull the $7$-arc toward each $0$-arc. 
In the process, 
we can assume that the $7$-arc crosses over several $a$-arcs 
for $a\in\{0,3,4,6,7,8,10\}$ missing $1$; 
for there is no crossing of $\{1|1|1\}$. 
Then we have a pair of new crossings of 
$$\{a|7|3-a\}=\{0|7|3\},\ \{3|7|0\}, \ \{4|7|10\}, \ \{6|7|8\}, \ \{7|7|7\}, \ \{8|7|6\}, \ \{10|7|4\},$$ 
respectively. 
We remark that the colors $3$ and $10$ appear 
at the crossings of $\{0|7|3\}$ and $\{4|7|10\}$.

Now, every crossing with a $1$-arc is of color $\{1|6|0\}$ or $\{1|0|10\}$. 
The endpoints of every $1$-arc are under-crossings of color 
\begin{itemize}
\item[(i)] 
$\{1|6|0\}$ both, 
\item[(ii)] 
$\{1|0|10\}$ both, or 
\item[(iii)] 
$\{1|0|10\}$ and $\{1|6|0\}$.
\end{itemize}
For every $1$-arc of type (i), 
we deform the diagram near the $1$-arc equipped with a $7$-arc 
into type (ii) 
as shown in the left of Figure~\ref{fig4-07}. 
Here, the colors $3$ and $10$ appear 
at the crossings of $\{0|7|3\}$, $\{7|3|10\}$, 
$\{0|3|6\}$, $\{3|3|3\}$, and $\{3|0|8\}$.

\begin{figure}[htb]
\begin{center}
\includegraphics[bb=0 0 335 179]{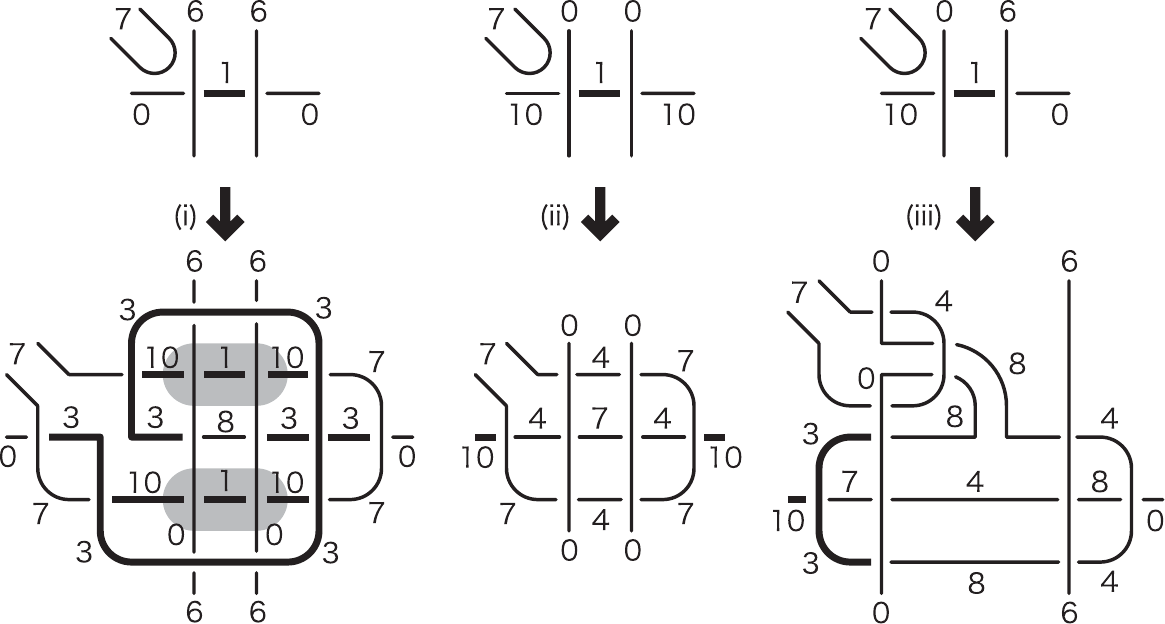}
\caption{}
\label{fig4-07}
\end{center}
\end{figure}

For every $1$-arc of type (ii) or (iii), 
we deform the diagram near the $1$-arc with a $7$-arc 
as shown in the center and right of the figure, 
so that we can remove all the $1$-arcs from the diagram. 
We remark that the colors $3$ and $10$ appear at the crossings of 
$\{4|7|10\}$ for (ii) and 
$\{3|0|8\}$ and $\{7|3|10\}$ for (iii). 

Since the original diagram has a $1$-arc, 
at least one of deformations (i), (ii), and (iii) must happen. 
Therefore, the obtained diagram has a $10$-arc. 
If the diagram has no $3$-arc, 
the case (ii) must happen. 
By deforming a neighborhood of a crossing of $\{4|0|7\}$ 
similarly to Figures~\ref{fig2-04} and \ref{fig2-05}, 
we can make a pair of crossings of $\{0|7|3\}$ 
so that we have 
${\rm Im}(C)=\{0,3,4,6,7,8,10\}$. 
\end{proof}

We remark that the $11$-colored diagram $(D,C)$ in Lemma~\ref{lem44} 
has no crossing of color $\{3|10|6\}$, $\{6|8|10\}$, or $\{10|10|10\}$. 
In particular, there is no crossing whose over-arc is colored $10$.


\section{Proof of Theorem~\ref{thm11}{\rm (ii)}--Part II}\label{sec5}

Let $G_1$ be the graph obtained from the palette graph 
$G(\{0,4,6,7,8\})$ by adding two vertices $3$ and $10$ and 
five edges 
$$\{0|3|6\}, \ \{0|7|3\}, \ \{3|0|8\}, \ 
\{4|7|10\}, \ \{7|3|10\}.$$
See Figure~\ref{fig5-01}. 
In other words, $G_1$ is obtained from 
$G(\{0,3,4,5,6,8,10\})$ 
by deleting the edges $\{3|10|6\}$ and $\{6|8|10\}$.

\begin{figure}[htb]
\begin{center}
\includegraphics[bb=0 0 123 77]{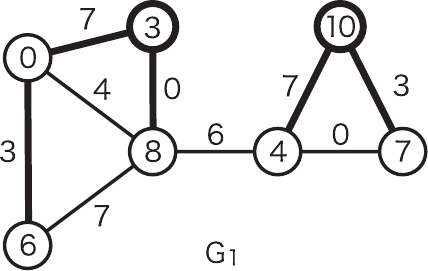}
\caption{}
\label{fig5-01}
\end{center}
\end{figure}

Assume that $(D,C)$ satisfies Lemma~\ref{lem44}. 
If $\{a|c|b\}$ is the non-trivial color of a crossing of $(D,C)$, 
then the palette graph $G_1$ has the corresponding edge $\{a|c|b\}$.

\begin{lemma}\label{lem51}
For any $11$-colorable knot $K$, 
there is an $11$-colored diagram $(D,C)$ of $K$ 
such that 
\begin{itemize}
\item[{\rm (i)}] 
${\rm Im}(C)=\{0,3,4,6,7,8\}$, and 
\item[{\rm (ii)}] 
there is no crossing of color $\{6|6|6\}$. 
\end{itemize}
\end{lemma}

\begin{proof}
We may assume that $(D,C)$ satisfies Lemma~\ref{lem44}. 
Since the graph $G_1$ has no edge whose label is $10$ 
and $(D,C)$ has no crossing of $\{10|10|10\}$, 
we see that there is no crossing of color $\{*|10|*\}$. 

Since there are two edges incident to the vertex $10$ in $G_1$, 
every crossing with a $10$-arc is of color 
$\{4|7|10\}$ or $\{7|3|10\}$. 
If there is a crossing of $\{7|3|10\}$, 
we deform the $7$-arc near the crossing as shown in the left of Figure~\ref{fig5-02} 
to replace the crossing with one of $\{4|7|10\}$. 
We remark that the crossings of $\{0|7|3\}$ and $\{4|0|7\}$ are also produced. 
Therefore, we may assume that there is no crossing of $\{7|3|10\}$.

\begin{figure}[htb]
\begin{center}
\includegraphics[bb=0 0 352 62]{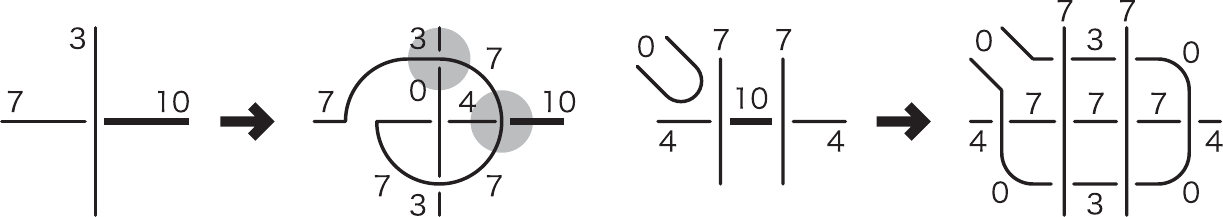}
\caption{}
\label{fig5-02}
\end{center}
\end{figure}

There is a $0$-arc in $(D,C)$. 
We will pull the $0$-arc toward each $10$-arc. 
In the process, we can assume that the $0$-arc crosses over several arcs 
whose colors are $0,3,4,7,8$ missing $6$ and $10$. 
In fact, since there is no crossing of color 
$$\{3|10|6\}, \ \{2|6|10\}, \ \{6|8|10\}, \ \{6|6|6\}, \ \{10|10|10\},$$
the set of $6$- and $10$-arcs is a disjoint union of intervals, 
and the complement in the plane is connected. 
When the $0$-arc crosses over an $a$-arc for $a=0,3,4,7,8$, 
we have a pair of new crossings of color 
$$\{a|0|-a\}=\{0|0|0\}, \ \{3|0|8\}, \ \{4|0|7\}, \ \{7|0|4\}, \ \{8|0|3\},$$
respectively. 
We remark that any vertex of $G_1$ other than $6$ and $10$ 
is $0$ itself or incident to an edge labeled $0$. 

We deform the diagram near every $10$-arc with a $0$-arc 
as shown in the right of the figure, 
so that we remove all the $10$-arcs from the diagram. 
We remark that the crossings of 
$\{0|7|3\}$, $\{4|0|7\}$, and $\{7|7|7\}$ are produced. 
\end{proof}


\section{Proof of Theorem~\ref{thm11}{\rm (ii)}--Part III}\label{sec6}

\begin{lemma}\label{lem61}
For any $11$-colorable knot $K$, 
there is an $11$-colored diagram $(D,C)$ of $K$ 
such that 
\begin{itemize}
\item[{\rm (i)}] 
${\rm Im}(C)=\{0,3,4,6,7,8\}$,  
\item[{\rm (ii)}] 
there is no crossing of color $\{3|3|3\}$, $\{4|4|4\}$, or $\{6|6|6\}$. 
\end{itemize}
\end{lemma}

\begin{proof}
We may assume that $(D,C)$ satisfies Lemma~\ref{lem51} 
with ${\rm Im}(C)=\{0,3,4,6,7,8\}$. 
Figure~\ref{fig6-01} shows the palette graph 
$G(\{0,3,4,6,7,8\})$, 
which is obtained from $G_1$ by deleting the vertex $10$ 
and its incident edges $\{4|7|10\}$ and $\{7|3|10\}$.

\begin{figure}[htb]
\begin{center}
\includegraphics[bb=0 0 123 77]{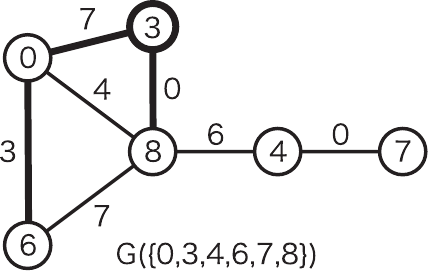}
\caption{}
\label{fig6-01}
\end{center}
\end{figure}

There is a $0$-arc in $(D,C)$. 
Similarly to the proof of Lemma~\ref{lem51}, 
we can pull the $0$-arc freely 
without producing new colors. 
We remark that any vertex of $G(\{0,3,4,6,7,8\})$ 
other than $6$ is $0$ itself or 
incident to an edge labeled $0$. 
Then we deform the diagram near every $3$- or $4$-arc 
with a $0$-arc as shown in Figure~\ref{fig6-02} 
so that there is no crossing of $\{3|3|3\}$ or $\{4|4|4\}$. 
\end{proof}

\begin{figure}[htb]
\begin{center}
\includegraphics[bb=0 0 343 55]{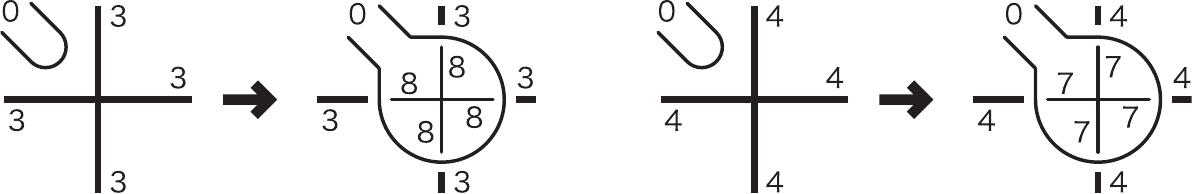}
\caption{}
\label{fig6-02}
\end{center}
\end{figure}

\begin{lemma}\label{lem62}
For any $11$-colorable knot $K$, 
there is an $11$-colored diagram $(D,C)$ of $K$ 
such that 
\begin{itemize}
\item[{\rm (i)}] 
${\rm Im}(C)=\{0,3,4,6,7,8\}$,  
\item[{\rm (ii)}] 
there is no crossing of color $\{*|3|*\}$, $\{4|4|4\}$, or $\{6|6|6\}$. 
\end{itemize}
\end{lemma}

\begin{proof}
We may assume that $(D,C)$ satisfies Lemma~\ref{lem61}. 
In the palette graph $G(\{0,3,4,6,7,8\})$, 
there is only one edge whose label is $3$. 
Therefore, 
every crossing whose over-arc is $3$ 
has the color $\{0|3|6\}$. 

There is a $7$-arc in $(D,C)$. 
We will pull the $7$-arc toward each crossing of $\{0|3|6\}$. 
Since there is no crossing of $\{4|4|4\}$, 
we can assume that the $7$-arc 
crosses over several arcs whose colors are 
$0,3,6,7,8$ missing $4$. 
If the $7$-arc crosses an $a$-arc for $a\in\{0,3,6,7,8\}$, 
then we have a pair of new crossings of color 
$$\{a|7|3-a\}=\{0|7|3\}, \ \{3|7|0\}, \ \{6|7|8\}, \ \{7|7|7\}, \ \{8|7|6\},$$
respectively. 
We remark that any vertex of $G(\{0,3,4,6,7,8\})$ other than $4$ 
is $7$ itself or incident to an edge labeled $7$. 
We deform the diagram near every crossing of $\{0|3|6\}$ 
equipped with a $7$-arc 
as shown in Figure~\ref{fig6-03} 
to remove all the crossings of $\{0|3|6\}$. 
\end{proof}

\begin{figure}[htb]
\begin{center}
\includegraphics[bb=0 0 154 55]{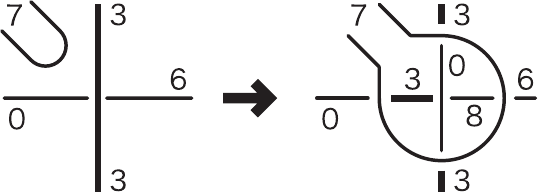}
\caption{}
\label{fig6-03}
\end{center}
\end{figure}

\begin{proof}[Proof of {\rm Theorem~\ref{thm11}(ii)}] 
We may assume that $(D,C)$ satisfies Lemma~\ref{lem62}. 
Since there are two edges incident to the vertex $3$ 
in $G(\{0,3,4,6,7,8\})$, 
every crossing with a $3$-arc is of color $\{3|0|8\}$ or $\{0|7|3\}$. 
Therefore, the endpoints of every $3$-arc are under-crossings of color 
\begin{itemize}
\item[(i)] 
$\{3|0|8\}$ and $\{0|7|3\}$, 
\item[(ii)] 
$\{3|0|8\}$ both, or 
\item[(iii)] 
$\{0|7|3\}$ both. 
\end{itemize}

For every $3$-arc of type (i), 
we deform the diagram near the crossing of $\{0|7|3\}$, 
which reduces a $3$-arc of type (ii). 
See the left of Figure~\ref{fig6-04}. 
Therefore, we may assume that there is no $1$-arc of type (i).

\begin{figure}[htb]
\begin{center}
\includegraphics[bb=0 0 307 143]{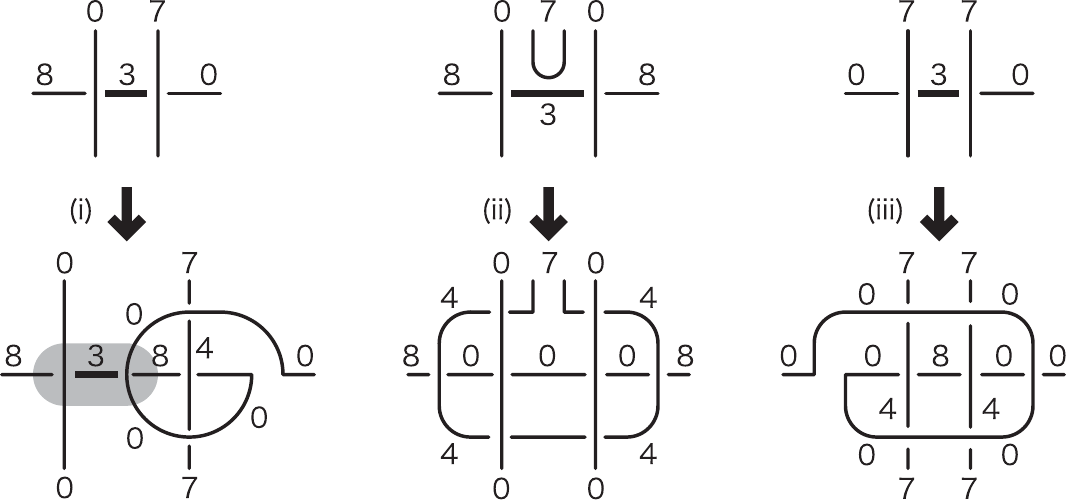}
\caption{}
\label{fig6-04}
\end{center}
\end{figure}

To remove a $3$-arc of type (ii), 
We will pull a $7$-arc toward the $3$-arc. 
Since there is no crossing of $\{4|4|4\}$, 
the $7$-arc can cross over several arcs whose colors are $0,3,6,7,8$ 
missing $4$ 
similarly to the proof of Lemma~\ref{lem62}. 
We remark that when the $7$-arc crosses over an $0$- or $3$-arc, 
then we have a pair of new crossings of color 
$\{0|7|3\}$. 
We deform the diagram near every $3$-arc of type (ii) with a $7$-arc 
to remove all the $3$-arcs of type (ii). 
See the center of the figure. 

Now, since every crossing with a $3$-arc is of color $\{0|7|3\}$, 
every $3$-arc is of type (iii). 
We deform the diagram near every $3$-arc of type (iii) with a $7$-arc 
as shown in the right of the figure 
so that we obtain a diagram with no $3$-arc. 
\end{proof}

\begin{corollary}\label{cor63}
For any $11$-colorable knot $K$ and $a\not\equiv b\in{\Z}/11{\Z}$, 
there is an $11$-colored diagram $(D,C)$ of $K$  
with $${\rm Im}(C)=\{a,b,5a+7b,2a+10b,10a+2b\}.$$
\end{corollary}

\begin{proof}
Let $f:{\Z}/11{\Z}\rightarrow{\Z}/11{\Z}$ be the affine map  
defined by $f(x)=3(b-a)x+a$. 
Since the map $f$ satisfies  
$$f(0)=a, \ f(4)=b, \ f(6)=5a+7b, \ f(7)=2a+10b, \ f(8)=10a+2b,$$
we have the conclusion by 
Lemma~\ref{lem210} and Theorem~\ref{thm11}(ii). 
\end{proof}


\section{$11$-colorable ribbon $2$-knot}\label{sec7}

A {\it ribbon $2$-knot} \cite{Fox} 
is a kind of knotted $2$-sphere embedded in ${\R}^4$. 
Such a $2$-knot is presented by a diagram in ${\R}^3$ 
with only double point circles \cite{Yaj}, 
the $n$-colorability is defined similarly to the classical case 
by assigning an element of ${\Z}/n{\Z}$ to each sheet of the diagram. 
Refer to \cite{CS} for a diagram of a knotted surfaces.

\begin{lemma}\label{lem71} 
Let $K$ be an $11$-colorable ribbon $2$-knot. 
For each set $S=\{1,4,6,7,8\}$ or $\{0,4,6,7,8\}$, 
there is an $11$-colored diagram of $K$ which satisfies the following. 
\begin{itemize}
\item[{\rm (i)}] 
Every double point circle has a neighborhood 
as shown in {\rm Figure~\ref{fig7-01}}, 
and all the sheets of the diagram other than the small shaded disks 
are colored by $S$. 
\item[{\rm (ii)}] 
While the color $2a-b$ of the shaded disk may not belong to $S$, 
the pair $(a,b)$ must satisfy $2b-a\in S$. 
\end{itemize}
\end{lemma}

\begin{proof} 
Let $A$ be a virtual arc which presents $K$ \cite{Sat}. 
Since $K$ is $11$-colorable,  
so is $A$. 
Then there is an $11$-colored diagram $(D,C)$ of $A$ 
with ${\rm Im}(C)=S$ by a similar argument 
in the proof of Theorems~\ref{thm11}. 
The diagram of $K$ associated to $(D,C)$ 
is the desired one \cite{Osh, Sat2}. 
\end{proof}

\begin{figure}[htb]
\begin{center}
\includegraphics[bb=0 0 253 91]{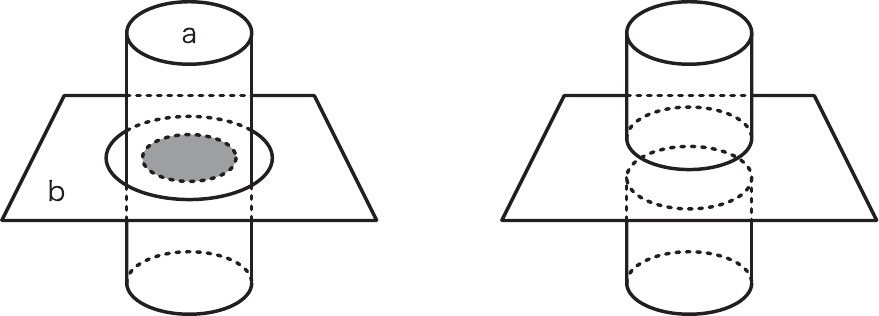}
\caption{}
\label{fig7-01}
\end{center}
\end{figure}

\begin{theorem}\label{thm72}
Any $11$-colorable ribbon $2$-knot satisfies the following. 
\begin{itemize}
\item[{\rm (i)}] 
There is an $11$-colored diagram $(D_1,C_1)$ of $K$  
with ${\rm Im}(C_1)=\{1,4,6,7,8\}$. 
\item[{\rm (ii)}] 
There is an $11$-colored diagram $(D_2,C_2)$ of $K$ 
with ${\rm Im}(C_2)=\{0,4,6,7,8\}$. 
\end{itemize}
\end{theorem}

\begin{proof} 
(i) 
We may assume that $(D,C)$ satisfies 
Lemma~\ref{lem71} for $S=\{1,4,6,7,8\}$. 
In the left of Figure~\ref{fig7-01}, 
the shaded disk is colored $2b-a$. 
The pair $(a,b)$ with $a,b,2b-a\in S$ and $2a-b\not\in S$ 
is one of the following: 
$$(a,b)= (4,1), \ (4,7), \ (1,7), \ (1,6), \ (7,6), \ (7,8), \
(6,8), \ (6,4), \ (8,4), \ (8,1).$$
In fact, each edge $\{x|y|z\}$ 
in the palette graph $G(S)$ 
produces such two pairs $(y,x)$ and $(y,z)$.

First, we consider the case $(a,b)=(4,1)$, 
where the shaded sheet is colored $9$. 
There is an $8$-sheet in $(D,C)$. 
We pull the $8$-sheet toward the $9$-sheet 
without introducing new double points 
and deform the diagram as shown in the left of Figure~\ref{fig7-02} 
to remove the $9$-sheet. 
We remark that the figure shows a cross-section 
of the neighborhood of the $9$-sheet. 
Next, we consider the case $(a,b)=(4,7)$, 
where the shaded sheet is colored $10$. 
We deform the horizontal $4$-sheet by surrounding the $10$-sheet, 
that reduces the case $(a,b)=(4,1)$. 
See the right of the figure.

\begin{figure}[htb]
\begin{center}
\includegraphics[bb=0 0 238 175]{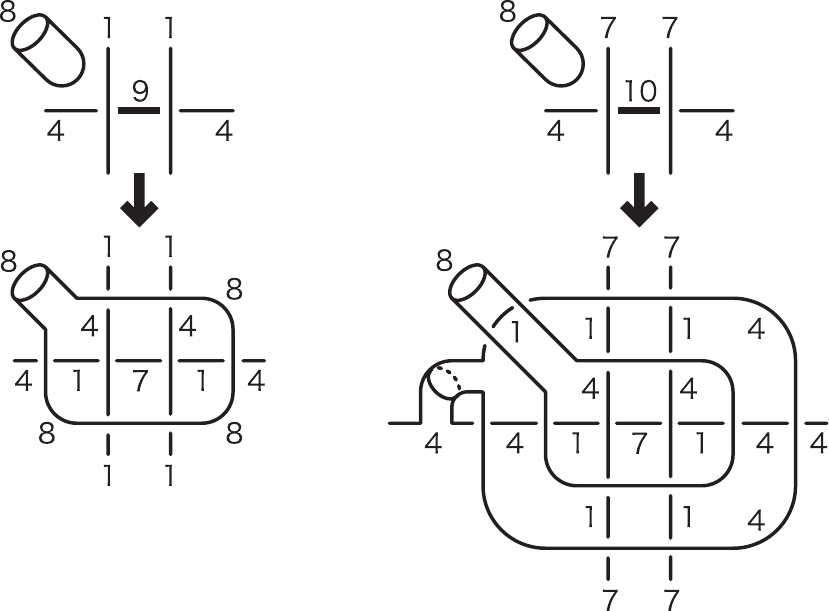}
\caption{}
\label{fig7-02}
\end{center}
\end{figure}

Let $f:{\Z}/11{\Z}\rightarrow{\Z}/11{\Z}$ be the affine map 
defined by $f(x)=9x+9$. 
Since we have  
$$f(1)=7, \ f(4)=1, \ f(6)=8, \ f(7)=6, \mbox{ and }f(8)=4,$$
the cases $(a,b)=(1,7)$, $(7,6)$, $(6,8)$, and $(8,4)$ 
are obtained from $(a,b)=(4,1)$ by applying $f$ 
repeatedly, 
and the cases $(a,b)=(1,6)$, $(7,8)$, $(6,4)$, and $(8,1)$ 
are obtained from $(a,b)=(4,7)$ similarly.

(ii) 
We may assume that $(D,C)$ satisfies Lemma~\ref{lem71} 
for $S=\{0,4,6,7,8\}$. 
The pair $(a,b)$ with $a,b,2b-a\in S$ and $2a-b\not\in S$ 
is one of the following: 
$$(a,b)=(4,8), \ (7,6), \ (7,8), \ (6,8), \ (6,4), \mbox{ and }(0,7).$$
In fact, each edge $\{x|y|z\}$ in the palette graph $G(S)$ 
produces such two pairs $(y,x)$ and $(y,z)$ 
other than $(4,8)$ from $\{0|4|8\}$ and $(0,4)$ from $\{4|0|7\}$. 

For the case $(a,b)=(4,8)$, 
we deform the horizontal $4$-sheet by surrounding the shaded $1$-sheet 
as shown in the left of Figure~\ref{fig7-03} 
so that we can remove the $1$-sheet. 
The case $(a,b)=(0,7)$ can be similarly proved. 
See the right of the figure.

\begin{figure}[htb]
\begin{center}
\includegraphics[bb=0 0 217 143]{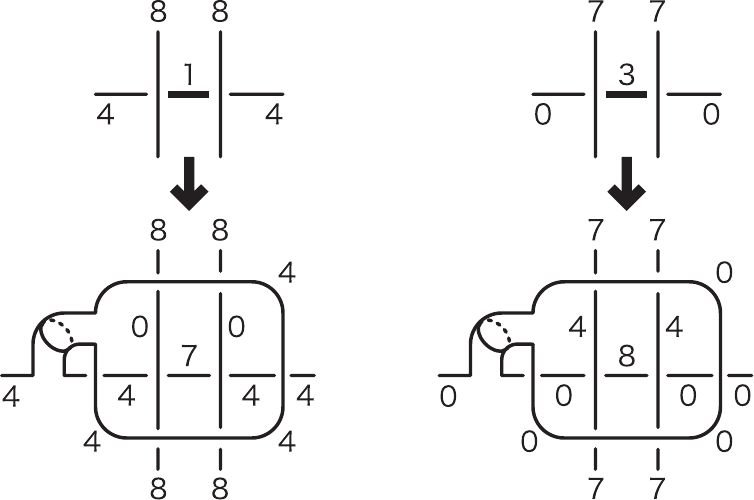}
\caption{}
\label{fig7-03}
\end{center}
\end{figure}

For the case $(a,b)=(7,6)$, 
we pull a $0$-sheet and deform the diagram as shown in the left of Figure~\ref{fig7-04}. 
Then we can remove the $5$-sheet without introducing new colors. 
For the case $(a,b)=(7,8)$, 
we first deform the horizontal $7$-sheet by surrounding the shaded $9$-sheet, 
that reduces to the case $(a,b)=(7,6)$.

\begin{figure}[htb]
\begin{center}
\includegraphics[bb=0 0 292 184]{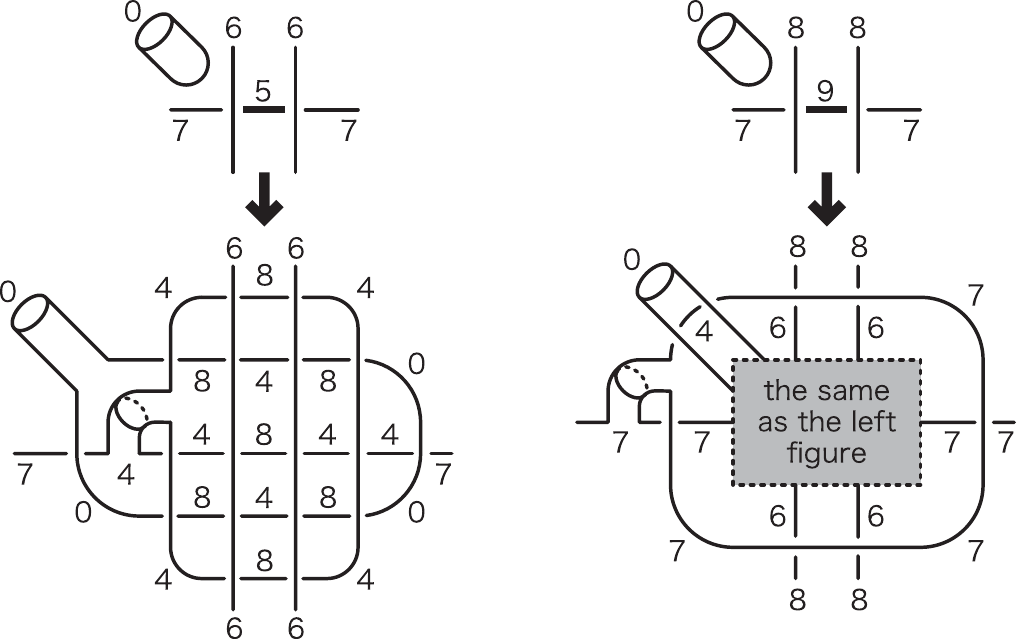}
\caption{}
\label{fig7-04}
\end{center}
\end{figure}

For the case $(a,b)=(6,8)$, 
we pull a $7$-sheet and surround the shaded $10$-sheet 
by the $7$-sheet as shown in the left of Figure~\ref{fig7-05} 
so that the color $10$ is removed. 
For the case $(a,b)=(6,4)$, 
we pull a $7$-sheet toward the shaded $2$-sheet 
and deform the horizontal $6$-sheet to surround the $2$-sheet. 
Then this case reduces to the case $(a,b)=(6,8)$. 
\end{proof}

\begin{figure}[htb]
\begin{center}
\includegraphics[bb=0 0 238 175]{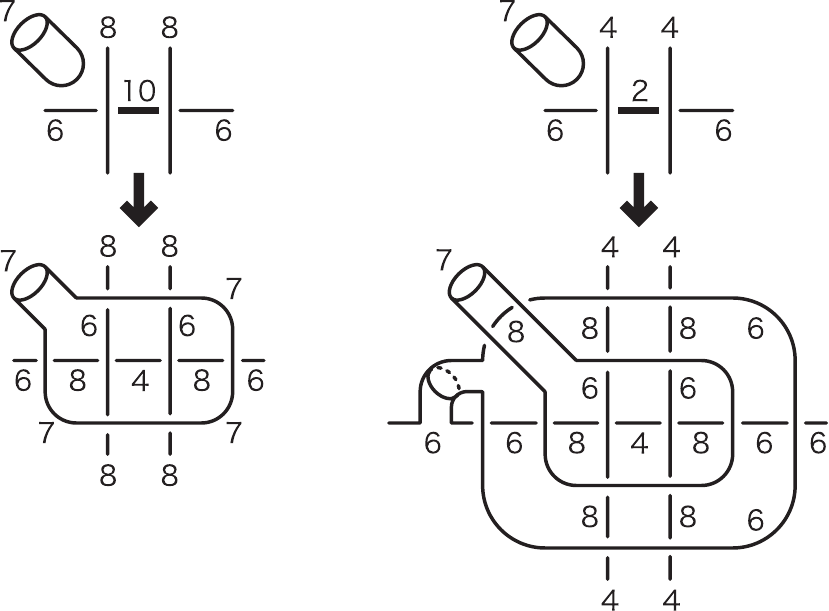}
\caption{}
\label{fig7-05}
\end{center}
\end{figure}

For a $p$-colorable $2$-knot $K$, 
we denote by ${\rm C}_p(K)$ the minimum number  of $\#{\rm Im}(C)$ 
for all non-trivially $p$-colored diagrams $(D,C)$ of $K$ 
\cite{Sat2}. 
Then the following is an immediate consequence of 
Theorem~\ref{thm72}.

\begin{corollary}\label{cor73} 
Any $11$-colorable ribbon $2$-knot $K$ satisfies 
${\rm C}_{11}(K)=5$. 
 \hfill$\Box$
\end{corollary}

The proof of the following is 
as same as that of Corollaries~\ref{cor32} and \ref{cor63}.

\begin{corollary}\label{cor74} 
For any $11$-colorable ribbon $2$-knot $K$ and $a\not\equiv b\in{\Z}/11{\Z}$, 
there are $11$-colored diagrams $(D_1,C_1)$ and $(D_2,C_2)$ of $K$ with 
$$
\begin{array}{l}
{\rm Im}(C_1)=\{a,b,3a+9b,10a+2b,6a+6b\}, \ and\\
{\rm Im}(C_2)=\{a,b,5a+7b,2a+10b,10a+2b\}.
\end{array}
$$
\hfill$\Box$
\end{corollary}


\end{document}